\documentclass{scrartcl}
\usepackage{amsthm}
\usepackage{amsfonts}
\usepackage{amssymb}
\usepackage{amsxtra}
\usepackage{amscd}

\numberwithin{equation}{section} 
\setkomafont{title}{\normalfont}

\newcommand{\coloneqq}{:=}
\newcommand{\ie}{i.e.}
\newenvironment{pdeq}{ \left\{ \begin{aligned}}{\end{aligned}\right.}

\newcommand{\np}[1]{(#1)}

\newcommand{\bp}[1]{\big(#1\big)}
\newcommand{\bb}[1]{\big[#1\big]}
\newcommand{\Bp}[1]{\bigg(#1\bigg)}

\newcommand{\cala}{{\mathcal A}}

\newcommand{\calf}{{\mathcal F}}

\newcommand{\caln}{{\mathcal N}}

\newcommand{\calp}{{\mathcal P}}

\newcommand{\cals}{{\mathcal S}}
\newcommand{\calt}{{\mathcal T}}

\newcommand{\R}{\mathbb{R}}

\newcommand{\N}{\mathbb{N}}
\DeclareMathOperator{\e}{e}

\DeclareMathOperator{\Div}{div}

\DeclareMathOperator*{\esssup}{ess\,sup}

\newcommand{\bogopr}{\mathfrak{B}}
\newcommand{\embeds}{\hookrightarrow}

\newcommand{\set}[1]{\ensuremath{\{#1\}}}

\newcommand{\setc}[2]{\ensuremath{\{#1\ \vert\ #2\}}}
\newcommand{\setcl}[2]{\ensuremath{\bigl\{#1\ \big\vert\ #2\bigr\}}}

\newcommand{\ball}{\mathrm{B}}
\renewcommand{\restriction}[2]{#1\big | _{#2}}
\newcommand{\proj}{\calp}
\newcommand{\projcompl}{\calp_\bot}
\newcommand{\Rn}{{\R^n}}
\newcommand{\grad}{\nabla}
\newcommand{\D}{{\mathrm D}}
\newcommand{\dx}{{\mathrm d}x}

\newcommand{\dt}{{\mathrm d}t}

\newcommand{\norm}[1]{\lVert#1\rVert}

\newcommand{\snorm}[1]{{\lvert #1 \rvert}}

\newcommand{\WSR}[2]{\mathrm{W}^{#1,#2}} 
\newcommand{\WSRN}[2]{\mathrm{W}^{#1,#2}_0} 
\newcommand{\DSR}[2]{\mathrm{D}^{#1,#2}} 
 
\newcommand{\WSRloc}[2]{\mathrm{W}^{#1,#2}_{\mathrm{loc}}} 
\newcommand{\CR}[1]{\mathrm{C}^{#1}}  
\newcommand{\LR}[1]{\mathrm{L}^{#1}}
\newcommand{\LRloc}[1]{\mathrm{L}^{#1}_{\mathrm{loc}}} 
\newcommand{\CRi}{\CR \infty}
\newcommand{\CRci}{\CR \infty_0}

\newcommand{\WSRhom}[2]{\mathrm{D}^{#1,#2}} 
\newcommand{\WSRhomN}[2]{\mathrm{D}^{#1,#2}_0} 
\newcommand{\LRper}[1]{\mathrm{L}^{#1}_{\mathrm{per}}}
\newcommand{\LRpercompl}[1]{\mathrm{L}^{#1}_{\mathrm{per},\bot}}
\newcommand{\WSRper}[2]{\mathrm{W}^{#1,#2}_{\mathrm{per}}} 
\newcommand{\WSRpercompl}[2]{\mathrm{W}^{#1,#2}_{\mathrm{per},\bot}}

\newcommand{\Xoseen}[1]{{\mathrm{X}^{#1}_{\rey}}}
\newcommand{\Zoseen}[1]{{\mathrm{Z}^{#1}_{\rey}}}
\newcommand{\vvel}{v}
\newcommand{\vpres}{p}
\newcommand{\Vvel}{V}

\newcommand{\wvel}{w}
\newcommand{\wpres}{\mathfrak{q}}

\newcommand{\uvel}{u}
\newcommand{\upres}{\mathfrak{p}}

\newcommand{\tin}{\text{in }}
\newcommand{\tif}{\text{if }}
\newcommand{\ton}{\text{on }}

\newcommand{\rey}{\lambda}
\newcommand{\per}{\calt}
\newcommand{\eone}{\e_1}

\newcommand{\cutoff}{\chi}

\theoremstyle{plain}
\newtheorem{thm}{Theorem}[section]

\newtheorem{lem}[thm]{Lemma}
\newtheorem{prop}[thm]{Proposition}

\theoremstyle{remark}
\newtheorem{rem}[thm]{Remark}

\begin{document}
%%%%%%%%%%%%%%%%%%%%%%%%%%%%%%%%%%%%%%%%%%%%%%%%%%%%%%%%%%%%%%
%%          Title, author, date, abstract, etc.             %%
%%%%%%%%%%%%%%%%%%%%%%%%%%%%%%%%%%%%%%%%%%%%%%%%%%%%%%%%%%%%%%
\title{New results for the Oseen problem with applications to the Navier--Stokes equations in exterior domains}

\author{
Thomas Eiter
and
Giovanni P. Galdi
}

\date{}
\maketitle

\begin{abstract}
We prove new existence and uniqueness results in full Sobolev spaces 
for the steady-state Oseen problem in a smooth exterior domain of $\Rn$, $n\ge 2$. These results are then employed, on the one hand, in the study of analogous properties for the corresponding (linear) time-periodic case and, on the other hand and more significantly, to prove analogous properties for their nonlinear counterpart, at least for small data.   
\par 
\end{abstract}


%%%%%%%%%%%%%%%%%%%%%%%%%%%%%%%%%%%%%%%%%%%%%%%%%%%%%%%%%%%%%%
%%          Main document                                   %%
%%%%%%%%%%%%%%%%%%%%%%%%%%%%%%%%%%%%%%%%%%%%%%%%%%%%%%%%%%%%%%
\section{Introduction}

As is well known, the steady-state Oseen problem consists in solving the following set of equations
\begin{align}\label{sys:Oseen}
\begin{pdeq}
-\Delta\uvel+\rey\partial_1\uvel+\grad\upres&= f 
& \tin \Omega, \\
\Div\uvel&=0
& \tin \Omega, \\
\uvel&=0
&\ton \partial\Omega, \\
\lim_{\snorm{x}\to\infty} u(x)&=0,
\end{pdeq}
\end{align}
where $\Omega$ is an exterior domain of $\Rn$, $f\colon\Omega\to\Rn$ and $\lambda$ ($>0$)  are  given external force and dimensionless (Reynolds) number,
whereas  $\uvel\colon\Omega\to\Rn$ and $\upres\colon\Omega\to\R$
are unknown velocity  and pressure fields, respectively. 
\par
Problem \eqref{sys:Oseen} has been investigated by a number of authors, beginning with the pioneering work \cite{GaO}; for a rather detailed, yet incomplete, list of contributors and corresponding contributions we refer the reader to  \cite[Chapter VII]{GaldiBookNew}, \cite{amr} and the bibliography there included. The peculiarity of these results  is due to the circumstance that the function space where $u$ belongs is {\em not} a full Sobolev space but, instead, a homogeneous Sobolev space. The latter means that for a given $f$ in the Lebesgue space $\LR{q_3}(\Omega)$ (suitable $q_3\in (1,\infty)$),  the associated velocity-field solution  $u$, its first derivatives and its second derivatives belong, in the order, to {\em different} $\LR{q_i}$-spaces, $i=1,2,3$, with 
\begin{equation}\label{1}
q_1>q_2>q_3.
\end{equation} 
These findings are sharp, in the sense that, under the stated assumptions on $f$, it can be shown by means of counterexamples that the numbers $q_1,q_2,q_3$ must, in general, be different and satisfy \eqref{1}. 
\par
However, particularly motivated by the recent approaches to the study of time-periodic well-posedness \cite{GaldiKyed_TPflowViscLiquidpBody,GaKy} and time-periodic bifurcation \cite{GaBi,GaBi1}, we would like to investigate which {\em further} assumptions  (if any) $f$ must satisfy in order to ensure that $u$ belongs to the full Sobolev space $\WSR{2}{q}(\Omega)$ ($q\equiv q_3$). A positive answer to this question would, for example, allow to frame time-periodic bifurcation in a full Sobolev space and hopefully, analyze the phenomenon of secondary bifurcation in a similar way as that employed for flow in bounded domains \cite{JI}. Notice that the rigorous interpretation of this phenomenon in the case of an exterior domain is, to date, an entirely open problem.     
\par
The main objective of this paper is to show that if, in addition to being in $\LR{q}$, $f$ is in the dual, $\WSRhomN{-1}{r}(\Omega)$, of a suitable homogeneous Sobolev space\footnote{See the next section for the precise definition of $\WSRhomN{-1}{r}(\Omega)$.}, with $r=r(q,n)$, then there exists a unique $(u,\mathfrak p)$ solving \eqref{sys:Oseen} with, in particular,  $u\in \WSR{2}{q}(\Omega),\nabla\mathfrak p\in \LR{q}(\Omega)$. 
Moreover, the solution depends continuously on $f$, uniformly in $\lambda\in (0,\lambda_0)$ for arbitrarily fixed $\lambda_0>0$; see Theorem \ref{thm:OseenLinear}. 
\par
In view of the results established in \cite{GaldiKyed_TPflowViscLiquidpBody}, Theorem \ref{thm:OseenLinear} produces an immediate corollary that ensures that a similar property holds also for time-periodic solutions of period $\calt>0$ to the (linear) Oseen problem (see \eqref{sys:OseenTP}), provided $f$ is periodic of the same period $\calt$; see Theorem  \ref{thm:OseenLinearTPfull}.
\par
Finally, combining Theorem \ref{thm:OseenLinear} and Theorem  \ref{thm:OseenLinearTPfull} with the contraction mapping theorem, we may extend the results proved there to the fully nonlinear Navier--Stokes case (see \eqref{sys:NavierStokesOseen}, \eqref{sys:NavierStokesOseenTP}), on condition that the ``size" of $f$ and $\lambda$ is suitably restricted, and $n\ge 3$. We thus show existence in full Sobolev space for both steady-state (Theorem \ref{thm:OseenNonlinear}) and time-periodic (Theorem \ref{thm:OseenNonlinearTP}) Navier--Stokes problems in exterior domain. 
\par
The plan of the paper is as follows. In Section 2, after recalling some basic notation used in the paper, we state and comment our main results. In the subsequent Section 3, we provide the proof of well-posedness for the  steady-state  (Theorem \ref{thm:OseenLinear}) and time-periodic (Theorem \ref{thm:OseenLinearTPfull}) Oseen problem. Finally, in Section 4 we extend the results of the preceding section to the fully nonlinear case; see Theorem \ref{thm:OseenNonlinear} and Theorem~\ref{thm:OseenNonlinearTP}.   

\section{Statement of the Main Results}
We begin to introduce our principal notation. Unless otherwise stated, by the symbol $\Omega$ we mean a (smooth) exterior domain of $\Rn$, \ie, the complement of a (smooth) compact set $\Omega_0$. With the origin of coordinates in the interior of $\Omega_0$ we put $\ball_R:=\{x\in \Rn:|x|<R\}$,    $\Omega_R\coloneqq\Omega\cap\ball_R$, and $\Omega^R\coloneqq\Omega\setminus\overline{\ball^R}$.

For $(t,x)\in\R\times\Omega$, we set $\partial_t\coloneqq\partial/\partial t$,  $\partial_j\coloneqq\partial/\partial{x_j}$, $j=1,\ldots,n$, and, as customary, for $\alpha\in\N_0^n$ we  set
$\D^\alpha\coloneqq\partial_{1}^{\alpha_1}\cdots\partial_{n}^{\alpha_n}$,
and denote by $\grad^k\uvel$ the collection of all  spatial
derivatives $\D^\alpha\uvel$ of $u$ of order $\snorm{\alpha}=k$.

For $\cala$ an open set of $\R^n$ and $q\in[1,\infty]$, we denote by
$\LR{q}(\cala)$ and
$\WSR{k}{q}(\cala)$  the classical Lebesgue and Sobolev spaces  of order $k\in\N$,
equipped with norms $\norm{\cdot}_{q}=\norm{\cdot}_{q;\cala}$ and $\norm{\cdot}_{k,q}=\norm{\cdot}_{k,q;\cala}$, respectively.
We also consider homogeneous Sobolev spaces:
\[
\WSRhom{k}{q}(\cala) \coloneqq
\setc{
\uvel\in\LRloc{1}(\cala)}{
\grad^k\uvel\in\LR{q}(\cala)
},
\]
with corresponding seminorm
\[
\snorm{\uvel}_{k,q}=
\snorm{\uvel}_{k,q;\cala}\coloneqq
\norm{\grad^k\uvel}_{q;\cala}\coloneqq
\sum_{\snorm{\alpha}=k}\norm{\D^\alpha\uvel}_{q;\cala}\,,
\]
and $\WSRhomN{k}{q}(\cala)$ obtained by (Cantor) completing
$\CRci(\cala)$
in the norm  $\snorm{\,\cdot\,}_{k,q}$.
We indicate the latter's  dual space by $\WSRhomN{-k}{q'}(\cala)$, 
where $q'=q/(q-1)$, $q\in(1,\infty)$,
with norm $\snorm{\,\cdot\,}_{-k,q'}$.

Let $X$ be a seminormed vector space, $\per>0$,  and $q\in[1,\infty]$. 
Then $\LRper{q}(\R;X)$ is the space of all measurable 
 $f\colon\R\to X$ such that $f(t+\per)=f(t)$ for almost all $t\in\R$ and
$\norm{f}_{\LRper{q}(\R;X)}<\infty$, 
where
\[
\norm{f}_{\LRper{q}(\R;X)}
\coloneqq 
\Bp{\frac{1}{\per}\int_0^\per \norm{f(t)}_X^q\,\dt}^\frac{1}{q} \quad \tif q<\infty, 
\qquad
\norm{f}_{\LRper{\infty}(\R;X)}
\coloneqq 
\esssup\limits_{t\in\R} \norm{f(t)}_{X}.
\]
For simplicity, we set
$\LRper{q}(\R\times\cala)\coloneqq\LRper{q}(\R;\LR{q}(\cala))$ and
$\norm{f}_{q}\coloneqq\norm{f}_{\LRper{q}(\R;\LR{q}(\cala))}$
for $f\in\LRper{q}(\R\times\cala)$.
Moreover, we introduce the ``maximal regularity space"
\[
\WSRper{1,2}{q}(\R\times\cala)
\coloneqq
\setcl{\uvel\in\LRper{q}(\R\times\cala)}
{\uvel\in\LRper{q}(\R;\WSR{2}{q}(\cala)),\ 
\partial_t\uvel\in\LRper{q}(\R\times\cala)},
\]
equipped with the norm
\[
\norm{\uvel}_{1,2,q}
\coloneqq\norm{\uvel}_{\WSRper{1,2}{q}(\R\times\cala)}
\coloneqq\norm{\uvel}_{\LRper{q}(\R;\WSR{2}{q}(\cala))}
+\norm{\partial_t\uvel}_{\LR{q}(\R\times\cala)}.
\]

For functions $f\in\LRper{q}(\R;X)$ 
we introduce the projections
\[
\proj f \coloneqq \frac{1}{\per}\int_0^\per f(t)\,\dt,
\qquad
\projcompl f\coloneqq f-\proj f,
\]
and we call $\proj f\in X$ the \emph{steady-state} part
and $\projcompl f \in \LRper{q}(\R;X)$ the \emph{purely oscillatory} part of $f$.
Setting $\LRpercompl{q}(\R;X)\coloneqq\projcompl\LRper{q}(\R;X)$,
we obtain the decomposition
\[
\LRper{q}(\R;X)=X\oplus\LRpercompl{q}(\R;X).
\]
We shall also  use the notation 
\[
\LRpercompl{q}(\R\times\cala)\coloneqq\projcompl\LRper{q}(\R\times\cala),
\qquad
\WSRpercompl{1,2}{q}(\R\times\cala)\coloneqq\projcompl\WSRper{1,2}{q}(\R\times\cala).
\]

Unless otherwise stated, we do not distinguish between the real-valued function space $X$
and its $\Rn$-valued analogue $X^n$, $n\in\N$. 

We use the letter $C$ to denote generic positive constants in our estimates.
The dependence of a constant $C$ on quantities $a,b,\ldots$, 
will be emphasized by  writing $C(a,b,\ldots)$.
\smallskip\par
We are now in a position to state our main findings. We begin with the following theorem that, in fact, represents the key result upon which all the others rely.  

\begin{thm}\label{thm:OseenLinear}
Let $\Omega\subset\Rn$, $n\ge2$, 
and let $q\in(1,\infty)$, $r\in\np{\frac{n+1}{n},n+1}$,  $0<\rey\leq\rey_0$.
Set $s\coloneqq\frac{(n+1)r}{n+1-r}$.
Then, for every $f\in\LR{q}(\Omega)^n\cap\WSRhomN{-1}{r}(\Omega)^n$ 
there exists a solution
\[
\uvel\in\WSRhom{2}{q}(\Omega)^n\cap\WSRhom{1}{r}(\Omega)^n\cap\LR{s}(\Omega)^n, \qquad 
\upres\in\WSRhom{1}{q}(\Omega)
\]
to \eqref{sys:Oseen}.
This solution satisfies 
$\partial_1\uvel\in\LR{q}(\Omega)^n$
and obeys the estimates
\begin{align}
\snorm{\uvel}_{1,r}
+\rey^{\frac{1+\delta}{n+1}}\norm{\uvel}_{\frac{(n+1)r}{n+1-r}}
&\leq 
 C
\rey^{-\frac{M}{n+1}}\snorm{f}_{-1,r},
\label{est:OseenLinearEstWeak}
\\
\snorm{\uvel}_{2,q}
+\rey\norm{\partial_1\uvel}_q
+\norm{\grad\upres}_q
&\leq 
 C
\bp{
\norm{f}_q
+\rey^{-\frac{M}{n+1}}\snorm{f}_{-1,r}
}
\label{est:OseenLinearEstStrong}
\end{align}
for some constant $ C= C\np{n,q,r,\Omega,\rey_0}>0$,
where
\begin{align}\label{eq:DefExponentRHSLinearOseen}
M=
\begin{cases}
2 &\tif \frac{n+1}{n}<r\leq\frac{n}{n-1}, \\
0 &\tif \frac{n}{n-1}<r<n, \\
1 &\tif n\leq r < n+1,
\end{cases}
\qquad
\delta=
\begin{cases}
1 &\tif n=r=2,\\
0 &\text{else}.
\end{cases}
\end{align}
In particular, if $s\leq q$, then $\uvel\in\WSR{2}{q}(\Omega)^n$ and
\begin{align}\label{est:OseenLinearEstFullNorm}
\begin{split}
\rey^{\frac{\np{1+\delta}\theta}{n+1}}\norm{\uvel}_{q}
+\rey^{\frac{\np{1+\delta}\theta}{2(n+1)}}\snorm{\uvel}_{1,q}
+\snorm{\uvel}_{2,q} 
\leq 
 C
\bp{
\norm{f}_q
+\rey^{-\frac{M}{n+1}}\snorm{f}_{-1,r}
}
\end{split}
\end{align}
where
\begin{align}\label{eq:DefThetaFullNormEst}
\theta\coloneqq\frac{qs}{n(q-s)+qs}=\frac{(n+1)qr}{n(n+1)(q-r)+qr}\in[0,1].
\end{align}
Moreover, if $\np{\uvel_1,\upres_1}$ is another solution to \eqref{sys:Oseen}
that belongs to the same function class as $\np{\uvel,\upres}$,
then $\uvel=\uvel_1$ and $\upres=\upres_1+c$ for some constant $c\in\R$. 

Additionally, if $r>\frac{n}{n-1}$, 
we can choose $\upres$ such that $\upres\in\LR{r}(\Omega)$. 
Then $\partial_1\uvel\in\DSR{-1}{r}(\Omega)$ and it holds
\begin{align}\label{est:OseenLinearEstPressure}
\norm{\upres}_r+\rey\snorm{\partial_1\uvel}_{-1,r}
\leq 
 C
\rey^{-\frac{M}{n+1}}\snorm{f}_{-1,r}.
\end{align}
\end{thm}
\begin{rem}\label{rem:range_q_FullEst}
Note that the condition $s\leq q$ is equivalent to $\frac{1}{q}\leq\frac{1}{r}-\frac{1}{n+1}$. 
Therefore, the assumption $r>\frac{n+1}{n}$ in Theorem \ref{thm:OseenLinear} 
implies the necessary condition
$
q>\frac{n+1}{n-1}\,.
$
\end{rem}\par
The first, important consequence of this theorem is presented in the next one concerning the corresponding linear time-periodic problem
\begin{align}\label{sys:OseenTP}
\begin{pdeq}
\partial_t\uvel-\Delta\uvel+\rey\partial_1\uvel+\grad\upres&= f 
& \tin \R\times\Omega, \\
\Div\uvel&=0
& \tin \R\times\Omega, \\
\uvel&=0
&\ton \R\times\partial\Omega, \\
\lim_{\snorm{x}\to\infty} \uvel(t,x)&=0, 
&t\in\R,
\end{pdeq}
\end{align}
with  $f\colon\R\times\Omega\to\Rn$ a  given time-periodic external force. 
Precisely, we will prove the following.

\begin{thm}\label{thm:OseenLinearTPfull}
Let $\Omega\subset\Rn$, $n\geq2$,
and let $q\in(1,\infty)$, $r\in\np{\frac{n+1}{n},n+1}$ and $0<\rey\leq\rey_0$.
Set $s\coloneqq\frac{(n+1)r}{n+1-r}$. Then,
for every $f\in\LR{q}(\R\times\Omega)^n$ with $\proj f\in\WSRhom{-1}{r}(\Omega)^n$
there is a solution $\np{\uvel,\upres}=\np{\vvel+\wvel,\vpres+\wpres}$ to \eqref{sys:OseenTP}
with
\begin{align*}
\vvel&\in\WSRhom{2}{q}(\Omega)^n\cap\WSRhom{1}{r}(\Omega)^n\cap\LR{s}(\Omega)^n, & 
\vpres&\in\WSRhom{1}{q}(\Omega), \\
\wvel&\in\WSRpercompl{1,2}{q}(\R\times\Omega)^n, &
\wpres&\in\LRpercompl{q}(\R;\WSRhom{1}{q}(\Omega)),
\end{align*}
which satisfies
\begin{align}\label{est:OseenLinearTPFull}
\begin{split}
\snorm{\vvel}_{1,r}
+\rey^{\frac{1+\delta}{n+1}}\norm{\vvel}_{s}
&\leq 
C
\rey^{-\frac{M}{n+1}}\snorm{\proj f}_{-1,r},
\\
\snorm{\vvel}_{2,q}
+\rey\norm{\partial_1\vvel}_q
+\norm{\grad\vpres}_q
&\leq 
 C
\bp{
\norm{\proj f}_q
+\rey^{-\frac{M}{n+1}}\snorm{\proj f}_{-1,r}
},
\\
\norm{\wvel}_{1,2,q}
+\norm{\grad\wpres}_q
&\leq C\norm{\projcompl f}_q
\end{split}
\end{align}
for a constant 
$ C= C\np{q,r,\Omega,\rey_0}>0$
and $M$ and $\delta$ as in \eqref{eq:DefExponentRHSLinearOseen}.
Moreover, if
$\np{\uvel_1,\upres_1}$
is another solution to \eqref{sys:OseenTP} 
that belongs to the same function class as $\np{\uvel,\upres}$,
then $\uvel=\uvel_1$ and $\upres=\upres_1+\upres_0$ for 
some $\per$-periodic function $\upres_0\colon\R\to\R$.

In particular, if $s\leq q$, 
then $\uvel\in\WSRper{1,2}{q}(\R\times\Omega)^n$ 
and 
\begin{align}\label{est:OseenLinearEstTPFullNormSteadyState}
\rey^{\frac{\np{1+\delta}\theta}{n+1}}\norm{\vvel}_{q}
+\rey^{\frac{\np{1+\delta}\theta}{2(n+1)}}\snorm{\vvel}_{1,q}
+\snorm{\vvel}_{2,q} 
\leq 
 C
\bp{
\norm{\proj f}_q
+\rey^{-\frac{M}{n+1}}\snorm{\proj f}_{-1,r}
}
\end{align}
where $\theta\in[0,1]$ is given in \eqref{eq:DefThetaFullNormEst}.
\end{thm}

\begin{rem}\label{rem:range_q_FullEst_1}
The observation made in Remark \ref{rem:range_q_FullEst} applies to Theorem \ref{thm:OseenLinearTPfull} as well.
\end{rem}

Next, combining the above theorems with the contraction mapping theorem, we are able to extend analogous results to the nonlinear case, under the assumption of ``small" data.\par More specifically, let us begin to consider the steady-state problem
\begin{align}\label{sys:NavierStokesOseen}
\begin{pdeq}
-\Delta\vvel+\rey\partial_1\vvel+\grad\vpres+\vvel\cdot\grad\vvel&= f 
& \tin \Omega, \\
\Div\vvel&=0
& \tin \Omega, \\
\vvel&=-\rey\eone
&\ton \partial\Omega, \\
\lim_{\snorm{x}\to\infty} \vvel(x)&=0.
\end{pdeq}
\end{align}
We shall prove the following result.
\begin{thm}\label{thm:OseenNonlinear}
Let $\Omega\subset\R^n$, $n\geq 3$, and let $q,r\in(1,\infty)$ with
\begin{align}\label{cond:range_qr_nonlinear}
q&\geq\frac{n}{3},
&
\frac{1}{3q}+\frac{1}{n+1}&\leq\frac{1}{r},
&
\frac{2}{q}-\frac{4}{n}&\leq\frac{1}{r},
&
\frac{2}{n+1}&\leq\frac{1}{r}
<\begin{cases}
\frac{n-1}{n} &\tif n=3,4, \\
\frac{n}{n+1} &\tif n\geq5.
\end{cases}
\end{align}
Then there is $\rey_0>0$ such that 
for all $0<\rey\leq\rey_0$ 
we may find $\varepsilon>0$ such that
for all $f\in\LR{q}(\Omega)\cap \WSRhomN{-1}{r}(\Omega)$ 
satisfying
$\norm{f}_q+\snorm{f}_{-1,r}\leq \varepsilon$
there exists a pair $(\vvel,\vpres)$ with
\[
\vvel\in\WSRhom{2}{q}(\Omega)\cap\WSRhom{1}{r}(\Omega)\cap\LR{\frac{(n+1)r}{n+1-r}}(\Omega),
\quad
\partial_1\vvel\in\LR{q}(\Omega),
\quad
\vpres\in\WSRhom{1}{q}(\Omega)
\]
satisfying \eqref{sys:NavierStokesOseen}.
In particular, if $s\leq q$, then $\vvel\in\WSR{2}{q}(\Omega)^n$.
\end{thm}
\begin{rem}
As in Remark \ref{rem:range_q_FullEst}, 
the additional assumption $s\leq q$ is equivalent to $\frac{1}{q}\leq\frac{1}{r}-\frac{1}{n+1}$.
Therefore, the upper bound in \eqref{cond:range_qr_nonlinear}
leads to the necessary conditions
\[
q>\frac{n(n+1)}{n^2-n-1}
\quad \tif n=3,4, \qquad
q>\frac{n+1}{n-1}
\quad \tif n\geq 5.
\]
\end{rem}

Likewise, consider the  problem
\begin{align}\label{sys:NavierStokesOseenTP}
\begin{pdeq}
\partial_t\vvel-\Delta\vvel+\rey\partial_1\vvel+\grad\vpres+\vvel\cdot\grad\vvel&= f 
& \tin \R\times\Omega, \\
\Div\vvel&=0
& \tin \R\times\Omega, \\
\vvel&=-\rey\eone
&\ton \R\times\partial\Omega, \\
\lim_{\snorm{x}\to\infty} \vvel(t,x)&=0,
&t\in\R.
\end{pdeq}
\end{align}
where $f$ is a suitably prescribed time-periodic function.
We shall prove the following.
\begin{thm}\label{thm:OseenNonlinearTP}
Let $\Omega\subset\R^n$, $n\geq 3$, and let 
$q,r\in(1,\infty)$ with
\begin{align}
\frac{n+2}{3}&\leq q\leq n+1, &
\frac{n(n+1)}{n^2-n-1}&< q,
\label{cond:range_q_nonlinearTP}
\\
\frac{2}{q}-\frac{4}{n}&\leq\frac{1}{r}\leq\frac{2}{q},
&
\frac{1}{q}+\frac{1}{n+1}&\leq\frac{1}{r}<
\begin{cases}
\frac{n-1}{n} &\tif n=3,4, \\
\frac{n}{n+1} &\tif n\geq5.
\end{cases}
\label{cond:range_r_nonlinearTP}
\end{align}
Then there is $\rey_0>0$ such that 
for all $0<\rey\leq\rey_0$ we can find $\varepsilon>0$ such that
for all $f\in\LRper{q}(\R\times\Omega)^n$ with $\proj f\in\WSRhomN{-1}{r}(\Omega)^n$ 
satisfying
$\norm{f}_q+\snorm{\proj f}_{-1,r}\leq \varepsilon$
there exists a unique solution 
\[
(\vvel,\vpres)\in\WSRper{1,2}{q}(\R\times\Omega)^n\times\LRper{q}(\R;\WSRhom{1}{q}(\Omega)), \quad 
\proj\vpres\in\LR{r}(\Omega)
\]
to \eqref{sys:NavierStokesOseenTP}.
\end{thm}

\begin{rem}\label{rem:range_q_nonlinearTP}
For $n=3$, condition \eqref{cond:range_q_nonlinearTP} yields $q\in(\frac{12}{5},4]$.
For $n\geq 4$, the second restriction in \eqref{cond:range_q_nonlinearTP} is redundant 
and it simplifies to $q\in(\frac{n+2}{3},n+1]$.
\end{rem}

\section{Proofs of Theorem \ref{thm:OseenLinear} and Theorem \ref{thm:OseenLinearTPfull}}
In order to prove Theorem \ref{thm:OseenLinear},
we first establish the following density result, 
which enables us 
to consider problem \eqref{sys:Oseen} only for  right-hand sides $f\in\CRci(\Omega)^n$.
\begin{prop}\label{prop:DensityInLqD-1r}
Let $\Omega\subset\Rn$ be an arbitrary domain and let $q,r\in(1,\infty)$.
Then $\CRci(\Omega)$ is a dense subset of $\LR{q}(\Omega)\cap\WSRhomN{-1}{r}(\Omega)$.
\end{prop}
\begin{proof}
The space $\LR{q}(\Omega)\cap\WSRhomN{-1}{r}(\Omega)$ 
can be identified with the dual space of $\LR{q'}(\Omega)+\WSRhomN{1}{r'}(\Omega)$, 
where $s'=s/(s-1)$.
Identifying elements of $\CRci(\Omega)$ with the corresponding functionals,
we consider 
$g\in\LR{q'}(\Omega)+\WSRhomN{1}{r'}(\Omega)$
that is an element of the kernel of each functional in $\CRci(\Omega)$,
\ie,
\[
\int_\Omega \varphi\,g\,\dx=0
\]
for all $\varphi\in\CRci(\Omega)$. 
This implies $g=0$. 
Consequently, by a standard duality argument, 
$\CRci(\Omega)$ is dense in $\LR{q}(\Omega)\cap\WSRhomN{-1}{r}(\Omega)$.
\end{proof}

We recall the notion of weak solutions: 
A pair $\np{\uvel,\upres}\in\WSRhomN{1}{r}(\Omega)^n\times\LRloc{r}(\Omega)$ 
is called \emph{weak solution} to \eqref{sys:Oseen} if 
$\Div\uvel=0$ and
\[
\int_\Omega \grad\uvel:\grad\varphi+\rey\partial_1\uvel\cdot\varphi\,\dx
=\int_\Omega\upres\,\Div\varphi + f\cdot \varphi \,\dx
\]
for all $\varphi\in\CRci(\Omega)^n$ with $\Div\varphi=0$.
We show that weak solutions have better regularity when $f$ is sufficiently regular.

\begin{lem}\label{lem:OseenStrongEstWithErrorTerms}
Let $\Omega\subset\Rn$ be an exterior domain of class $\CR{2}$. 
Let $q,r,s\in(1,\infty)$, $f\in\CRci(\Omega)^n$,
and let $\np{\uvel,\upres}\in\bp{\WSRhom{1}{r}(\Omega)^n\cap\LR{s}(\Omega)^n}\times\LRloc{r}(\Omega)$ 
be a weak solution to \eqref{sys:Oseen}.
Then $\uvel\in\WSRhom{2}{q}(\Omega)^n$, 
$\partial_1\uvel\in\LR{q}(\Omega)^n$ and 
$\upres\in\WSRhom{1}{q}(\Omega)$, 
and for each $R>0$ with $\partial\Omega\subset\ball_R$
there exists 
$C
= C(n,q,\Omega,R)>0$
such that
\begin{align}\label{est:OseenLinearWithErrorTerms}
\snorm{\uvel}_{2,q}
+\rey\norm{\partial_1\uvel}_q
+\snorm{\upres}_{1,q}
\leq C\np{1+\rey^4}\bp{
\norm{f}_q
+\norm{\uvel}_{q;\Omega_R}
+\norm{\upres}_{q;\Omega_R}
}.
\end{align}
\end{lem}
\begin{proof}
By \cite[Theorem VII.1.1]{GaldiBookNew}, we have
$\uvel\in\WSRloc{2}{q}(\Omega)\cap\CRi(\Omega)$ and $\upres\in\WSRloc{1}{q}(\Omega)\cap\CRi(\Omega)$. 
Let $0<R_0<R_1<R$ such that $\partial\ball_{R_0}\subset\Omega$, 
and let $\cutoff\in\CRci(\ball_{R_1})$ with $\cutoff\equiv 1$ on $\ball_{R_0}$. 
We set $\vvel\coloneqq\np{1-\cutoff}\uvel+\bogopr\np{\uvel\cdot\grad\cutoff}$, 
where $\bogopr$ denotes the Bogovski\u{\i} operator,
and $\vpres\coloneqq\np{1-\cutoff}\upres$.
Then $\vvel\in\WSRloc{2}{q}(\Rn)\cap\WSRhom{1}{r}(\Rn)\cap\LR{s}(\Rn)$
and $\vpres\in\WSRloc{1}{q}(\Rn)$ satisfy
\begin{align}\label{sys:OseenWholeSpace}
\begin{pdeq}
-\Delta\vvel+\rey\partial_1\vvel+\grad\vpres&= F 
& \tin \Rn, \\
\Div\vvel&=0
& \tin \Rn
\end{pdeq}
\end{align}
with
\[
F
\coloneqq\np{1-\cutoff}f
+2\grad\cutoff\cdot\grad\uvel
+\Delta\cutoff\uvel
-\rey\partial_1\cutoff\uvel
+\upres\grad\cutoff
+\bb{-\Delta+\rey\partial_1}\bogopr\np{\uvel\cdot\grad\cutoff}.
\]
By \cite[Theorem VII.7.1]{GaldiBookNew},
there exists a solution $\np{\vvel_1,\vpres_1}\in\WSRhom{2}{q}(\Omega)\times\WSRhom{1}{q}(\Omega)$ 
to \eqref{sys:OseenWholeSpace} that satisfies
\begin{align}\label{est:OseenWholeSpace}
\snorm{\vvel_1}_{2,q}
+\rey\norm{\partial_1\vvel_1}_q
+\snorm{\vpres_1}_{1,q}
\leq C\norm{F}_q.
\end{align}
Now set $\wvel\coloneqq\vvel-\vvel_1$. 
Then $\wvel$ is a solution to the homogeneous Oseen system
in the whole space.
Therefore, $\wvel=\vvel-\vvel_1$ is a polynomial,
which can be readily shown with the help of Fourier transform.
From \cite[Theorem VII.6.1]{GaldiBookNew} and $f\in\CRci(\Omega)$ 
we conclude $\D^\alpha\uvel(x)\to 0$ and thus $\D^\alpha\vvel(x)\to 0$ as $\snorm{x}\to\infty$
for each $\alpha\in\N_0^n$.
In virtue of $\D^\alpha\vvel_1\in\LR{q}(\Omega)$ for $\snorm{\alpha}=2$, 
the polynomial $\D^\alpha\wvel=\D^\alpha\vvel-\D^\alpha\vvel_1$ must thus be zero, \ie,
$\D^\alpha\wvel=0$ for $\snorm{\alpha}=2$. 
In the same way we conclude $\partial_1\wvel=0$ and, in consequence, $\grad\wpres=0$.
Hence we can replace $\np{\vvel_1,\vpres_1}$ 
with $\np{\vvel,\vpres}$ in estimate \eqref{est:OseenWholeSpace}.
Since $\uvel=\vvel$ and $\upres=\vpres$ on $\Omega^{R_1}$,
estimate \eqref{est:OseenWholeSpace} thus implies
\begin{align}\label{est:OseenLinearEstFarAway}
\begin{split}
\snorm{\uvel}_{2,q;\Omega^{R_1}}
+\rey\norm{\partial_1\uvel}_{q;\Omega^{R_1}}
+\snorm{\upres}_{1,q;\Omega^{R_1}}
&\leq\snorm{\vvel}_{2,q}
+\rey\norm{\partial_1\vvel}_q
+\snorm{\vpres}_{1,q}\\
&\leq C\bp{
\norm{f}_q
+\np{1+\rey}\norm{\uvel}_{1,q;\Omega_{R_1}}
+\norm{\upres}_{q;\Omega_{R_1}}
}.
\end{split}
\end{align}
To derive the estimate near the boundary, we use another cut-off function 
$\cutoff_1\in\CRci(\ball_R)$ with $\cutoff_1\equiv1$ on $\ball_{R_1}$,
and we set $\vvel\coloneqq\cutoff_1\uvel$
and $\vpres\coloneqq\cutoff_1\upres$.
Then $\np{\vvel,\vpres}\in\WSR{2}{q}(\Omega)\times\WSR{1}{q}(\Omega)$ 
is a solution to
\[
\begin{pdeq}
-\Delta\vvel+\grad\vpres&= 
\cutoff_1 f 
-2\grad\cutoff_1\cdot\grad\uvel-\Delta\cutoff_1\uvel 
-\cutoff_1\rey\partial_1\uvel+\upres\grad\cutoff_1
& \tin \Omega_R, \\
\Div\vvel&=\uvel\cdot\grad\cutoff_1
& \tin \Omega_R, \\
\vvel&=0 
& \ton \partial\Omega_R.
\end{pdeq}
\]
It is well known (see \cite[Exercise IV.6.3]{GaldiBookNew} for example)
that then $\np{\vvel,\vpres}$ is subject to the estimate
\[
\norm{\vvel}_{2,q}+\norm{\grad\vpres}_q
\leq C\bp{
\norm{f}_{q;\Omega_R}
+\np{1+\rey}\norm{\uvel}_{1,q;\Omega_R}
+\norm{\upres}_{q;\Omega_R}
}.
\]
Since $\uvel=\vvel$ and $\upres=\vpres$ on $\Omega_{R_1}$, 
a combination of this estimate with \eqref{est:OseenLinearEstFarAway} yields
\[
\snorm{\uvel}_{2,q}
+\rey\norm{\partial_1\uvel}_q
+\snorm{\upres}_{1,q}
\leq C\np{1+\rey}\bp{
\norm{f}_q
+\np{1+\rey}\norm{\uvel}_{1,q;\Omega_R}
+\norm{\upres}_{q;\Omega_R}
}.
\]
Finally, an application of Ehrling's inequality
\[
\snorm{\uvel}_{1,q;\Omega_R}
\leq C\np{
\varepsilon^{-1}\norm{\uvel}_{q;\Omega_R}
+\varepsilon\snorm{\uvel}_{2,q;\Omega_R}
}
\]
(see \cite[Theorem 5.2]{AdamsFournier_SobolevSpaces_2ndEd})
for $\varepsilon>0$ sufficiently small leads to \eqref{est:OseenLinearWithErrorTerms}.
\end{proof}

\begin{proof}[Proof of Theorem \ref{thm:OseenLinear}]
For the moment, consider $f\in\CRci(\Omega)$.
The existence of a weak solution $\np{\uvel,\upres}$ to \eqref{thm:OseenLinear} with 
$\uvel\in\WSRhom{1}{r}(\Omega)\cap\LR{s}(\Omega)$
satisfying \eqref{est:OseenLinearEstWeak}
follows from \cite[Theorem 2.2]{KimKim_LqEstStatOseen}. 
In the case $q>\frac{n}{n-1}$, one shows $\upres\in\LR{r}(\Omega)$ and
$\norm{\upres}_r\leq C\rey^{-\frac{M}{n+1}}\snorm{f}_{-1,r}$
in the same way as in \cite[Proof of Theorem VII.7.2]{GaldiBookNew}.
Then, for $\varphi\in\CRci(\Omega)^n$ we have
\[
\rey \int_\Omega \partial_1\uvel\cdot\varphi\,\dx 
=\int_\Omega \bp{
f \cdot \varphi
+\upres\Div\varphi
-\grad\uvel:\grad\varphi
}\,\dx
\leq \bp{\snorm{f}_{-1,r}
+\norm{\upres}_r
+\norm{\grad\uvel}_r}
\norm{\grad\varphi}_{r'},
\]
where $r'=r/(r-1)$,
which implies $\partial_1\uvel\in\WSRhomN{-1}{r}(\Omega)$ and
\[
\rey\snorm{\partial_1\uvel}_{-1,r}
\leq C\bp{\snorm{f}_{-1,r}
+\norm{\upres}_r
+\norm{\grad\uvel}_r}
\leq C\rey^{-\frac{M}{n+1}}\snorm{f}_{-1,r}.
\]
This shows \eqref{est:OseenLinearEstPressure} if $r>\frac{n}{n-1}$.
If this is not the case, we instead obtain a local estimate in the following way: 
First of all, by \cite[Lemma VII.1.1]{GaldiBookNew} we have $\upres\in\LR{r}(\Omega_R)$
for all $R>0$ with $\partial\ball_R\subset\Omega$.
For fixed $R$, we can add a constant to $\upres$ such that $\int_{\Omega_R}\upres=0$.
Now let $\psi\in\WSRN{1}{r'}(\Omega_R)^n$, $r'=r/(r-1)$, be a solution to the problem
\[
\Div\psi
=\snorm{\upres}^{r-2}\upres
-\frac{1}{\snorm{\Omega_R}}\int_{\Omega_R}\snorm{\upres}^{r-2}\upres\,\dx
=: g 
\qquad \tin \Omega_R,
\]
which exists 
since $g$ has vanishing mean value and satisfies $g\in\LR{r'}(\Omega_R)$
(see \cite[Theorem III.3.6]{GaldiBookNew} for example).
Moreover, we have
\[
\norm{\psi}_{1,r';\Omega_R}
\leq C\norm{g}_{r';\Omega_R}
\leq C\norm{\upres}_{r;\Omega_R}^{r-1}.
\]
Since $\np{\uvel,\upres}$ is a weak solution and $\upres$ has vanishing mean value on $\Omega_R$, 
we deduce
\begin{align*}
\norm{\upres}_{r;\Omega_R}^r
&=\int_{\Omega_R}\upres\Div\psi\,\dx
+\int_{\Omega_R}\upres\,\dx\ \frac{1}{\snorm{\Omega_R}}\int_{\Omega_R}\snorm{\upres}^{r-2}\upres\,\dx\\
&=\int_{\Omega_R} \grad\uvel:\grad\psi-\rey\partial_1\uvel\cdot\psi-f\cdot\psi\,\dx\\
&\leq C\np{1+\rey_0}\bp{\norm{\grad\uvel}_{r}+\norm{f}_{-1,r;\Omega_R}}\norm{\psi}_{1,r';\Omega_R}\\
&\leq C\np{1+\rey_0}\bp{\norm{\grad\uvel}_{r}+\snorm{f}_{-1,r}}\norm{\upres}_{r;\Omega_R}^{r-1}.
\end{align*}
Using estimate \eqref{est:OseenLinearEstWeak}, this leads to
\begin{align}\label{est:OseenLinearEstPressureLocal}
\norm{\upres}_{r;\Omega_R}
\leq C\rey^{-\frac{M}{n+1}}\snorm{f}_{-1,r}.
\end{align}
Next, by Lemma \ref{lem:OseenStrongEstWithErrorTerms},
from $f\in\CRci(\Omega)$ we conclude 
$\uvel\in\WSRhom{2}{q}(\Omega)$ and $\upres\in\WSRhom{1}{q}(\Omega)$ and the validity 
of \eqref{est:OseenLinearWithErrorTerms}.
We apply the estimate 
\[
\norm{\uvel}_{q;\Omega_R}
\leq C(\varepsilon)\norm{\uvel}_{\sigma;\Omega_R}
+\varepsilon\snorm{\uvel}_{1,q;\Omega_R}
\]
for $\varepsilon>0$ and $\sigma\in(1,\infty)$ several times to deduce
\begin{align*}
\norm{\uvel}_{q;\Omega_R}
&\leq  C(\varepsilon)\norm{\uvel}_s
+ C(\varepsilon)\snorm{\uvel}_{1,r}
+\varepsilon\snorm{\uvel}_{2,q},
\\
\norm{\upres}_{q;\Omega_R}
&\leq  C(\varepsilon)\norm{\upres}_{r;\Omega_R}
+\varepsilon\snorm{\upres}_{1,q}.
\end{align*}
Choosing $\varepsilon>0$ sufficiently small and 
combining these with the estimates \eqref{est:OseenLinearWithErrorTerms},
\eqref{est:OseenLinearEstWeak} and \eqref{est:OseenLinearEstPressureLocal}, 
we conclude \eqref{est:OseenLinearEstStrong} for $f\in\CRci(\Omega)$.
Employing the above estimates and Proposition \ref{prop:DensityInLqD-1r},
we can finally extend the result to general $f\in\LR{q}(\Omega)\cap\WSRhomN{-1}{r}(\Omega)$
by a standard density argument.

Moreover, the additional assumption $s\leq q$ yields the embedding
\[
\LR{s}(\Omega)\cap\WSRhom{2}{q}(\Omega)\embeds\WSR{2}{q}(\Omega),
\]
so that $\uvel\in\WSR{2}{q}(\Omega)$ in this case,
and the Gagliardo--Nirenberg inequality
(see \cite{MaremontiCrispo_InterpolationInequalityExteriorDomains})
implies
\[
\norm{\uvel}_{q}
\leq C\norm{\uvel}_{s}^\theta
\snorm{\uvel}_{2,q}^{1-\theta}
\leq C\rey^{-\frac{\np{1+\delta}\theta}{n+1}}
\bp{\norm{f}_q+\rey^{-\frac{M}{n+1}}\snorm{f}_{-1,r}}
\]
and
\[
\snorm{\uvel}_{1,q}
\leq C\norm{\uvel}_q^{1/2}\snorm{\uvel}_{2,q}^{1/2}
\leq C\rey^{-\frac{\np{1+\delta}\theta}{2(n+1)}}
\bp{\norm{f}_q+\rey^{-\frac{M}{n+1}}\snorm{f}_{-1,r}},
\]
where we used \eqref{est:OseenLinearEstStrong}.
This shows estimate \eqref{est:OseenLinearEstFullNorm}
and completes the proof.
\end{proof}

Now let us turn to the time-periodic Oseen problem \eqref{sys:OseenTP}.
We recall the following result, which treats the case $\proj f=0$.

\begin{thm}\label{thm:OseenLinearTPPurelyOsc}
Let $\Omega\subset\Rn$, $n\geq 2$, be an exterior domain of class $\CR{2}$, $q\in(1,\infty)$ and 
$\rey\in[0,\rey_0]$, $\rey_0>0$.
For any $f\in\LRpercompl{q}(\R\times\Omega)^n$
there is a solution 
\[
\np{\uvel,\upres}\in\WSRpercompl{1,2}{q}(\R\times\Omega)^n\times\LRpercompl{q}(\R;\WSRhom{1}{q}(\Omega))
\]
to \eqref{sys:OseenTP},
which satisfies
\begin{align}\label{est:OseenLinearTPPurelyOsc}
\norm{\uvel}_{1,2,q}
+\norm{\grad\upres}_q
\leq C\norm{f}_q
\end{align}
for a constant 
$ C= C\np{n,q,\Omega,\rey_0}>0$.
If 
$\np{\uvel_1,\upres_1}\in\WSRpercompl{1,2}{q}(\R\times\Omega)^n\times\LRpercompl{q}(\R;\WSRhom{1}{q}(\Omega))$
is another solution to \eqref{sys:OseenTP}, then $\uvel=\uvel_1$ and $\upres=\upres_1+\upres_0$ for 
some $\per$-periodic function $\upres_0\colon\R\to\R$.
\end{thm}
\begin{proof}
The result for $n=3$ has been established in \cite[Theorem 5.1]{GaldiKyed_TPflowViscLiquidpBody}.
The general case $n\geq 2$ is proved along the same lines.
\end{proof}

A combination of Theorem \ref{thm:OseenLinear} and Theorem \ref{thm:OseenLinearTPPurelyOsc}
allows us to treat general time-periodic forcing terms and
immediately leads to a proof of Theorem \ref{thm:OseenLinearTPfull}.

\begin{proof}[Proof of Theorem \ref{thm:OseenLinearTPfull}]
Set $f_1\coloneqq\proj f$ and $f_2\coloneqq\projcompl f$.
Let
$\np{\vvel,\vpres}\in\bp{\LR{s}\cap\WSRhom{2}{q}(\Omega)}\times\WSRhom{1}{q}(\Omega)$ 
be a solution to \eqref{sys:Oseen} with right-hand side $f=f_1$
that exists due to Theorem \ref{thm:OseenLinear}.
Moreover, let
$\np{\wvel,\wpres}\in\WSRpercompl{1,2}{q}(\R\times\Omega)\times
\LRpercompl{q}(\R;\WSRhom{1}{q}(\Omega))$ be a solution to \eqref{sys:OseenTP} 
with right-hand side $f=f_2$
that exists due to Theorem \ref{thm:OseenLinearTPPurelyOsc}.
Then $\np{\uvel,\upres}\coloneqq\np{\vvel+\wvel,\vpres+\wpres}$
is a solution to \eqref{sys:OseenTP}
with the asserted properties.
The uniqueness statement is deduced in a similar way.
\end{proof}

\section{Proofs of Theorem \ref{thm:OseenNonlinear} and Theorem \ref{thm:OseenNonlinearTP}}

In the following we focus on the time-periodic case 
and the proof of Theorem \ref{thm:OseenNonlinearTP}.
The proof of Theorem \ref{thm:OseenNonlinear} is very similar but less involved,
and we will sketch it at the end of this section.

First,
we reformulate \eqref{sys:NavierStokesOseenTP} 
as a problem with homogeneous boundary conditions.
For this purpose, let $R>0$ with $\partial\ball_R\subset\Omega$, 
and let $\varphi\in\CRci(\Rn)$ with $\varphi\equiv 1$ on $\ball_R$.
We define the function $\Vvel\colon\Rn\to\Rn$ by
\[
\Vvel(x)=\frac{\rey}{2}\bb{-\Delta+\grad\Div}\np{\varphi(x) x_2^2\eone}.
\]
Then $\Div\Vvel\equiv 0$ and $\Vvel(x)=-\rey\eone$ for $x\in\partial\Omega$,
and $\Vvel$ obeys the estimate
\begin{align}\label{est:BoundaryLiftTP}
\norm{-\Delta\Vvel+\rey\partial_1\Vvel}_q
+\snorm{-\Delta\Vvel+\rey\partial_1\Vvel}_{-1,r}
\leq C\rey\np{1+\rey}.
\end{align}
We set $\uvel(t,x)\coloneqq\vvel(t,x)-\Vvel(x)$, $\upres\coloneqq\vpres$. 
Then $(\vvel,\vpres)$ is a $\per$-time-periodic solution to \eqref{sys:NavierStokesOseenTP}
if and only if $(\uvel,\upres)$ is a $\per$-time-periodic solution to
\begin{align}\label{sys:NavierStokesOseenTPAfterLifting}
\begin{pdeq}
\partial_t\uvel-\Delta\uvel+\rey\partial_1\uvel+\grad\upres
&= f + \caln(\uvel)
& \tin \R\times\Omega, \\
\Div\uvel&=0
& \tin \R\times\Omega, \\
\uvel&=0
&\ton \R\times\partial\Omega, \\
\lim_{\snorm{x}\to\infty} \uvel(x)&=0,
\end{pdeq}
\end{align}
where 
\[
\caln(\uvel)
=
-\uvel\cdot\grad\uvel-\uvel\cdot\grad\Vvel-\Vvel\cdot\grad\uvel-\Vvel\cdot\grad\Vvel
+\Delta\Vvel-\rey\partial_1\Vvel.
\]
We will show existence of a solution to \eqref{sys:NavierStokesOseenTPAfterLifting}
in the function space 
\begin{align*}
\Xoseen{q,r}&\coloneqq\setcl{\uvel\in\WSRper{1,2}{q}(\R\times\Omega)^n}
{\Div\uvel=0,\,\restriction{\uvel}{\R\times\partial\Omega}=0,\,\norm{\proj\uvel}_{\rey}<\infty},
\\
\norm{\vvel}_{\rey}
&\coloneqq\snorm{\vvel}_{2,q}
+\snorm{\vvel}_{1,r}
+\rey^{\frac{1}{n+1}}\norm{\vvel}_s, \qquad 
s\coloneqq\frac{(n+1)r}{n+1-r},
\end{align*} 
which we equip with the norm 
\[
\norm{\uvel}_{\Xoseen{q,r}}\coloneqq
\norm{\proj\uvel}_{\rey}+\norm{\projcompl\uvel}_{1,2,q}.
\]
Then 
$\Xoseen{q,r}$ is a Banach space since $s\leq q$ by \eqref{cond:range_r_nonlinearTP}.
The following lemma enables us to derive suitable estimates for $\caln(\uvel)$
when $\uvel\in\Xoseen{q,r}$.

\begin{lem}\label{lem:NonlinearEstimatesTP}
Let $q,r\in(1,\infty)$ satisfy \eqref{cond:range_q_nonlinearTP} and \eqref{cond:range_r_nonlinearTP},
$0<\rey\leq\rey_0$, and let 
$\uvel_1,\uvel_2\in\Xoseen{q,r}$.
Set $\vvel_j\coloneqq\proj\uvel_j$, $\wvel_j\coloneqq\projcompl\uvel_j$ for $j=1,2$.
Then
\begin{align}
\norm{\vvel_1\cdot\grad\vvel_2}_q 
&\leq C\rey^{-\frac{\theta}{n+1}}\norm{\vvel_1}_{\rey}\norm{\vvel_2}_{\rey}, \label{est:NonlinearEstimateSteadyStateStrong}
\\
\snorm{\vvel_1\cdot\grad\vvel_2}_{-1,r}
&\leq C\rey^{-\frac{\eta}{n+1}}\norm{\vvel_1}_{\rey}\norm{\vvel_2}_{\rey},
\label{est:NonlinearEstimateSteadyStateWeak}
\\
\norm{\wvel_1\cdot\grad\wvel_2}_{q}
&\leq C\norm{\wvel_1}_{1,2,q}\norm{\wvel_2}_{1,2,q},
\label{est:NonlinearEstimatePurelyPerStrong}
\\
\snorm{\proj\np{\wvel_1\cdot\grad\wvel_2}}_{-1,r}
&\leq C\norm{\wvel_1}_{1,2,q}\norm{\wvel_2}_{1,2,q},
\label{est:NonlinearEstimatePurelyPerWeak}
\\
\norm{\vvel_1\cdot\grad\wvel_2}_{q}
&\leq C\rey^{-\frac{\zeta}{n+1}}\norm{\vvel_1}_{\rey}\norm{\wvel_2}_{1,2,q},
\label{est:NonlinearEstimateMixed1}
\\
\norm{\wvel_1\cdot\grad\vvel_2}_{q}
&\leq C\rey^{-\frac{\zeta}{n+1}}\norm{\wvel_1}_{1,2,q}\norm{\vvel_2}_{\rey},
\label{est:NonlinearEstimateMixed2}
\end{align}
where $\zeta\in[0,1)$ and $\theta,\eta\in[0,2]$. 
Moreover, $\eta=2$ if and only if $r=(n+1)/2$.
\end{lem}
\begin{proof}
Due to \eqref{cond:range_r_nonlinearTP}, 
the Gagliardo--Nirenberg inequality (see \cite{MaremontiCrispo_InterpolationInequalityExteriorDomains}) 
implies
\[
\norm{\vvel_1}_{3q}\leq C\snorm{\vvel_1}_{2,q}^{\theta_1} \norm{\vvel_1}_s^{1-\theta_1},
\qquad
\norm{\vvel_2}_{3q/2}\leq C\snorm{\vvel_2}_{2,q}^{\theta_2} \norm{\vvel_2}_s^{1-\theta_2},
\]
for
$\theta_1,\,\theta_2\in[0,1]$.
An application of H\"older's inequality thus yields 
\[
\norm{\vvel_1\cdot\grad\vvel_2}_q 
\leq\norm{\vvel_1}_{3q}\norm{\grad\vvel_2}_{3q/2}
\leq C\snorm{\vvel_1}_{2,q}^{\theta_1}\norm{\vvel_1}_s^{1-\theta_1}
\snorm{\vvel_2}_{2,q}^{\theta_2}\norm{\vvel_2}_{s}^{1-\theta_2}
\leq C\rey^{-\frac{\theta}{n+1}}\norm{\vvel_1}_{\rey}\norm{\vvel_2}_{\rey},
\]
which is \eqref{est:NonlinearEstimateSteadyStateStrong} with $\theta=2-\theta_1-\theta_2$.
Since $\frac{1}{q}-\frac{2}{n}\leq\frac{1}{2r}\leq\frac{1}{s}$, 
in the same way one shows \eqref{est:NonlinearEstimateSteadyStateWeak} by estimating
\[
\snorm{\vvel_1\cdot\grad\vvel_2}_{-1,r}
=\snorm{\Div\np{\vvel_1\otimes\vvel_2}}_{-1,r}
\leq C\norm{\vvel_1}_{2r}\norm{\vvel_2}_{2r}
\leq C\rey^{-\frac{\eta}{n+1}}\norm{\vvel_1}_{\rey}\norm{\vvel_2}_{\rey}.
\]
Note that \eqref{cond:range_r_nonlinearTP} implies $\eta\in[0,2]$.
To derive \eqref{est:NonlinearEstimatePurelyPerStrong},
we distinguish two different cases. 
On the one hand, if $q>\max\set{2,n/2}$, H\"older's inequality and 
the embedding theorem from \cite[Theorem 4.1]{GaldiKyed_TPflowViscLiquidpBody} 
yield
\[
\norm{\wvel_1\cdot\grad\wvel_2}_{q}
\leq\norm{\wvel_1}_{\LRper{q}(\R;\LR{\infty}(\Omega))}
\norm{\grad\wvel_2}_{\LRper{\infty}(\R;\LR{q}(\Omega))}
\leq C\norm{\wvel_1}_{1,2,q}\norm{\wvel_2}_{1,2,q}.
\]
On the other hand, if $(n+2)/3\leq q<(n+1)/2$, we conclude in the same way
\[
\norm{\wvel_1\cdot\grad\wvel_2}_{q}
\leq\norm{\wvel_1}_{\LRper{2q}(\R;\LR{\frac{nq}{n+1-2q}}(\Omega))}
\norm{\grad\wvel_2}_{\LRper{2q}(\R;\LR{\frac{nq}{2q-1}}(\Omega))}
\leq C\norm{\wvel_1}_{1,2,q}\norm{\wvel_2}_{1,2,q}.
\]
This yields \eqref{est:NonlinearEstimatePurelyPerStrong}.
Since \eqref{cond:range_q_nonlinearTP} and \eqref{cond:range_r_nonlinearTP} 
imply $\frac{1}{r} \geq \frac{2(n+2)}{nq}-\frac{6}{n}$, 
and we have
$\frac{1}{q}-\frac{2}{n}\leq\frac{1}{2r}\leq\frac{1}{q}$, 
for the derivation of \eqref{est:NonlinearEstimatePurelyPerWeak}
we can again use H\"older's inequality and \cite[Theorem 4.1]{GaldiKyed_TPflowViscLiquidpBody} 
to deduce 
\begin{align*}
\begin{split}
\snorm{\proj\np{\wvel_1\cdot\grad\wvel_2}}_{-1,r}
&=\snorm{\Div\proj\np{\wvel_1\otimes\wvel_2}}_{-1,r}
\leq C\norm{\wvel_1\otimes\wvel_2}_{\LRper{1}(\R;\LR{r}(\Omega))}
\\
&\leq C\norm{\vvel}_{\LRper{2}(\R;\LR{2r}(\Omega))}\norm{\wvel}_{\LRper{2}(\R;\LR{2r}(\Omega))}
\leq C\norm{\vvel}_{1,2,q}\norm{\wvel}_{1,2,q}.
\end{split}
\end{align*}
The remaining estimates \eqref{est:NonlinearEstimateMixed1} and \eqref{est:NonlinearEstimateMixed2}
follow in a similar fashion.
\end{proof}

\begin{proof}[Proof of Theorem \ref{thm:OseenNonlinearTP}]
It suffices to show existence of a solution to \eqref{sys:NavierStokesOseenTPAfterLifting}.
Consider the solution operator 
\[
\cals_\rey\colon\bp{\LR{q}(\Omega)^n\cap\WSRhomN{-1}{r}(\Omega)^n}\oplus
\LRpercompl{q}\np{\R\times\Omega}^n
\to\Xoseen{q,r}, \quad
f\mapsto \uvel,
\] 
where $\uvel$ is the unique velocity field of the solution 
$(\uvel,\upres)\in\Xoseen{q,r}\times\LRper{q}(\R;\WSRhom{1}{q}(\Omega))$ 
to \eqref{sys:OseenTP}
that exists due to Theorem \ref{thm:OseenLinearTPfull}.
This yields a family of continuous linear operators 
with 
\begin{align}\label{est:LinearOperatorTP}
\norm{\cals_\rey f}_{\Xoseen{q,r}}
\leq C
\bp{\norm{f}_q+\rey^{-\frac{M}{n+1}}\snorm{\proj f}_{-1,r}},
\end{align}
with $M$ as in \eqref{eq:DefExponentRHSLinearOseen},
compare estimate \eqref{est:OseenLinearTPFull}.
Then $(\uvel,\upres)$ is a solution to \eqref{sys:NavierStokesOseenTPAfterLifting}
if
$\uvel$ is a fixed point of the mapping
\[
\calf\colon \Xoseen{q,r}\to \Xoseen{q,r}, \quad \uvel\mapsto\cals_\rey\np{f+\caln(\uvel)}.
\]
Now consider 
$
\uvel\in 
A_\rho\coloneqq
\setcl{\uvel\in\Xoseen{q,r}}
{\norm{\uvel}_{\rey}\leq\rho}
$
for a radius $\rho>0$ that will be chosen below,
and set $\vvel\coloneqq\proj\uvel$, $\wvel\coloneqq\projcompl\uvel$.
Then we have
\begin{align*}
\proj\caln(\uvel)
&=-\vvel\cdot\grad\vvel
-\proj\np{\wvel\cdot\grad\wvel}
-\vvel\cdot\grad\Vvel
-\Vvel\cdot\grad\vvel
-\Vvel\cdot\grad\Vvel
+\Delta\Vvel
-\rey\partial_1\Vvel,
\\
\projcompl\caln(\uvel)
&=
-\vvel\cdot\grad\wvel
-\wvel\cdot\grad\vvel
-\projcompl\np{\wvel\cdot\grad\wvel}
-\wvel\cdot\grad\Vvel
-\Vvel\cdot\grad\wvel,
\end{align*}
and an application of estimates \eqref{est:LinearOperatorTP} and \eqref{est:BoundaryLiftTP}
together with Lemma \ref{lem:NonlinearEstimatesTP} leads to
\begin{align*}
\norm{\calf(\uvel)}_{\Xoseen{q,r}}
&\leq C\bp{
\norm{f+\caln(\uvel)}_{q}
+\rey^{-\frac{M}{n+1}}\snorm{\proj\np{f+\caln(\uvel)}}_{-1,r}
}\\
&\leq C\Big(
\norm{f}_q+\rey^{-\frac{M}{n+1}}\snorm{f}_{-1,r}
+\np{1+\rey^{-\frac{M}{n+1}}}\np{\rey+\rey^2}\\
&\qquad\quad+\bp{1+\rey^{-\frac{\theta}{n+1}}+\rey^{-\frac{\zeta}{n+1}}+\rey^{-\frac{M+\eta}{n+1}}}
\bp{\norm{\uvel}_{\Xoseen{q,r}}+\norm{\Vvel}_{\Xoseen{q,r}}}^2
\Big) \\
%&\leq C
%\bp{\np{1+\rey^{-\frac{M}{n+1}}}\np{\varepsilon+\rey}
%+\bp{1+\rey^{-\frac{\theta}{n+1}}+\rey^{-\frac{\zeta}{n+1}}+\rey^{-\frac{M+\eta}{n+1}}}\np{\rho+\rey}^2}
%\\
&\leq C
\bp{\rey^{-\frac{M}{n+1}}\np{\varepsilon+\rey}
+\bp{\rey^{-\frac{\theta}{n+1}}+\rey^{-\frac{\zeta}{n+1}}+\rey^{-\frac{M+\eta}{n+1}}}\np{\rho+\rey}^2}.
\end{align*}
Similarly, for $\uvel_1,\uvel_2\in A_\rho$ we obtain
\begin{align*}
\norm{\calf(\uvel_1)&-\calf(\uvel_2)}_{\Xoseen{q,r}}
\leq C\bp{
\norm{\caln(\uvel_1)-\caln(\uvel_2)}_{q}
+\rey^{-\frac{M}{n+1}}\snorm{\proj\np{\caln(\uvel_1)-\caln(\uvel_2)}}_{-1,r}
}\\
&\leq C\bp{1+\rey^{-\frac{\theta}{n+1}}+\rey^{-\frac{\zeta}{n+1}}+\rey^{-\frac{M+\eta}{n+1}}}
\np{\norm{\uvel_1}_{\rey}+\norm{\uvel_2}_{\Xoseen{q,r}}+\norm{\Vvel}_{\Xoseen{q,r}}}
\norm{\uvel_1-\uvel_2}_{\Xoseen{q,r}}
\\
%&\leq C
%\bp{1+\rey^{-\frac{\theta}{n+1}}+\rey^{-\frac{\zeta}{n+1}}+\rey^{-\frac{M+\eta}{n+1}}}
%\np{\rho+\rey}\norm{\uvel_1-\uvel_2}_{\Xoseen{q,r}}.
%\\
&\leq C
\bp{\rey^{-\frac{\theta}{n+1}}+\rey^{-\frac{\zeta}{n+1}}+\rey^{-\frac{M+\eta}{n+1}}}
\np{\rho+\rey}\norm{\uvel_1-\uvel_2}_{\Xoseen{q,r}}.
\end{align*}
Note that the assumptions imply $\max\set{\theta,\zeta,M+\eta}<n+1-M$,
so that we can consider $\gamma\in\R$ with
\[
1\leq\frac{n+1}{n+1-M}
<\gamma
<\frac{n+1}{\max\set{\theta,\zeta,M+\eta}}.
\]
Now we choose $\rey=\varepsilon=\rho^\gamma$
and $\rho>0$ so small that
\[
 C
\bp{\rho^{\gamma-\frac{\gamma M}{n+1}}
+\rho^{2-\gamma\frac{\theta}{n+1}}
+\rho^{2-\gamma\frac{\zeta}{n+1}}
+\rho^{2-\gamma\frac{M+\eta}{n+1}}}
\leq \rho,
\quad
 C
\bp{\rho^{1-\gamma\frac{\theta}{n+1}}
+\rho^{1-\gamma\frac{\zeta}{n+1}}
+\rho^{1-\gamma\frac{M+\eta}{n+1}}}
\leq\frac{1}{2}.
\]
This ensures that $\calf\colon A_\rho\to A_\rho$ is a contractive self-mapping, 
and the contraction mapping principle finally yields the existence of a fixed point of $\calf$.
This completes the proof.
\end{proof}

\begin{proof}[Proof of Theorem \ref{thm:OseenNonlinear}]
We may proceed in a similar way as in the previous proof. 
Here we introduce the function space
\begin{align*}
\Zoseen{q,r}&\coloneqq\setcl{\uvel\in\WSRhom{2}{q}(\Omega)^n\cap\WSRhom{1}{r}(\Omega)^n\cap\LR{s}(\Omega)^n}
{\Div\uvel=0,\,\restriction{\uvel}{\partial\Omega}=0},
\\
\norm{\uvel}_{\Zoseen{q,r}}
&\coloneqq\norm{\uvel}_{\rey}
\coloneqq\snorm{\uvel}_{2,q}
+\snorm{\uvel}_{1,r}
+\rey^{\frac{1}{n+1}}\norm{\uvel}_s, \qquad 
s\coloneqq\frac{(n+1)r}{n+1-r},
\end{align*} 
and the solution operator 
$
\cals_\rey\colon\bp{\LR{q}(\Omega)^n\cap\WSRhomN{-1}{r}(\Omega)^n}
\to\Zoseen{q,r}
$,
$f\mapsto \uvel$,
where $\uvel$ is the unique velocity field of a solution 
$(\uvel,\upres)$ 
to \eqref{sys:Oseen}
that exists due to Theorem \ref{thm:OseenLinear}.
Then $\np{\vvel,\vpres}$ is a solution to \eqref{sys:NavierStokesOseen} 
if and only if $\np{\uvel,\upres}\coloneqq\np{\vvel-\Vvel,\vpres}$ is a 
fixed-point of the mapping
\[
\calf\colon\Zoseen{q,r}\to\Zoseen{q,r}, \quad
\uvel\mapsto\cals_\rey\np{f+\caln(\uvel)},
\]
where $\Vvel$ and $\caln(\uvel)$ are given as before.
The existence of such a fixed point can then be shown as in the proof of 
Theorem \ref{thm:OseenNonlinearTP} by making use of estimates
\eqref{est:NonlinearEstimateSteadyStateStrong} and
\eqref{est:NonlinearEstimateSteadyStateWeak}.
\end{proof}
\bigskip\par\noindent
{\textbf{Acknowledgment.}} The work of G.P.~Galdi is partially supported by NSF grant DMS-1614011.

%%%%%%%%%%%%%%%%%%%%%%%%%%%%%%%%%%%%%%%%%%%%%%%%%%%%%%%%%%%%%%
%%          Bibliography                                    %%
%%%%%%%%%%%%%%%%%%%%%%%%%%%%%%%%%%%%%%%%%%%%%%%%%%%%%%%%%%%%%%
\bibliographystyle{abbrv}

\smallskip\par\noindent
Fachbereich Mathematik\\
Technische Universit\"at Darmstadt\\
Schlossgartenstr. 7, 64289 Darmstadt, Germany\\
Email: {\texttt{eiter@mathematik.tu-darmstadt.de}
\medskip\smallskip\par\noindent
Department of Mechanical Engineering and Materials Science\\
University of Pittsburgh\\
Pittsburgh, PA 15261, USA\\
Email: {\texttt{galdi@pitt.edu}}
\end{document}